\documentclass[11pt]{amsart}
\usepackage{amsfonts}
\usepackage{amscd}
\usepackage{amssymb}
\usepackage{amsmath}
\usepackage{enumerate}
\usepackage{epsfig}
\RequirePackage{color}

\usepackage{tikz}
\usepackage{tikz-cd}

\usetikzlibrary{matrix,arrows}

\newtheorem{theorem}{Theorem}[section]
\newtheorem{lemma}[theorem]{Lemma}

\newcommand{\mfr}[1]{\mathfrak{#1}}
\newcommand{\mr}[1]{\mathrm{#1}}

\newcommand{\ZZ}{\mathbb{Z}}
\newcommand{\FF}{\mathbb{F}}

\newcommand{\QQ}{\mathbb{Q}}

\newcommand{\PP}{\mathbb{P}}
\newcommand{\C}{\mathbb{C}}

\newcommand{\Proj}{\mathbf{P}}

\newcommand{\Jac}{\mathrm{Jac}}

\begin{document}
\title{Genus 3  curves whose Jacobians have endomorphisms by  $\QQ (\zeta _7 +\bar{\zeta}_7 )$ }

\author{J. William Hoffman}

\address{Department of Mathematics \\
              Louisiana State University \\
              Baton Rouge, Louisiana 70803}

 \author{Zhibin Liang}
\address{School of Mathematical Sciences \\
Capital Normal  University \\
Beijing, China, 100048\\
and Beijing International Center for Mathematical Research,\\
 Peking University
}

 \author{Yukiko Sakai}
\address{Department of Mathematics,
College of Liberal Arts and Sciences,
Kitasato University,
Kanagawa 252-0373
JAPAN}

 \author{Haohao Wang}
\address{Department of Mathematics \\
Southeast Missouri State University \\
Cape Girardeau, MO 63701}

\email{hoffman@math.lsu.edu,liangzhb@gmail.com,  y-sakai@kitasato-u.ac.jp,  hwang@semo.edu}


\begin{abstract}
In this work we consider constructions of genus three curves $X$
such that $\mr{End}(\mr{Jac} (X))\otimes \QQ$ contains the totally real
cubic number field $\QQ (\zeta _7 +\bar{\zeta}_7 )$. We construct explicit two-dimensional 
families defined over $\QQ (s, t)$ whose generic member is a nonhyperelliptic genus 3 curve with this 
property. The case  when $X$ is hyperelliptic was
studied in \cite{HW}. We calculate the zeta function of one of these curves. Conjecturally this zeta function
is described by a modular form.
\end{abstract}

\maketitle

\section{Introduction}
\label{S:intro}
%\subsection{} \label{SS:intro1}

Let $\mfr{M}_g$ be the (coarse) moduli space of projective smooth
curves of genus $g$. Recently there has been a lot of progress in
understanding the loci in $\mfr{M}_g$ defined by those curves $X$
such that the Jacobian variety $\mr{Jac}(X)$ has special
endomorphisms,  when $g=2$ (see, e.g., \cite{EK}, \cite{dG}, \cite{bR}, \cite{yS}).
The situation for $g\ge 3$ is much less studied.
 Recall that for any
polarized abelian variety $A$, $\mr{End}(A)\otimes \QQ$ is a
semisimple algebra of finite dimension with involution, and the
explicit types possible can be listed (see \cite[\S X]{aW} and
\cite{dM}). For a generic curve $X$, we have $\mr{End}(\Jac
(X))\otimes \QQ = \QQ$ so special endomorphisms means:
$\mr{End}(\Jac  (X))\otimes \QQ $ contains, but is larger than
$\QQ$. Throughout this paper we work with varieties over fields of
characteristic $0$.

Let $\mfr{A}_3$ be the moduli space of principally polarized abelian varieties of dimension 3. 
The map $X \mapsto \mr{Jac}(X): \mfr{M}_3 \to \mfr{A}_3$ is a birational equivalence
of 6-dimensional varieties. However not every genus 3  curve is hyperelliptic. The general
genus 3 curve is isomorphic via the canonical embedding to a smooth projective plane quartic.
The locus of hyperelliptic curves $\mfr{M}_3 ^{\mr{hyper}}\subset \mfr{M}_3$ is a 5-dimensional
irreducible subvariety. A hyperelliptic genus 3 curve can be given by an equation
$y^2 = f(x)$ where the polynomial $f(x)$ is of degree 7 or 8 and has distinct roots.
In \cite{HW} we showed via the method of Humbert and Mestre, that given a Poncelet 7-gon, there is a genus 3 hyperelliptic
curve $X$ canonically constructed from it, which had the property that
$\mr{End}(\mr{Jac}(X))\otimes \QQ\supset \QQ (\zeta _7 + \overline{\zeta}_7):=
 \QQ (\zeta _7 ^+ )$,
where $\zeta _7$ is a primitive 7th root of unity.

In this paper we give a construction of nonhyperelliptic curves
of  genus 3 with this property via a generalization of the method of Humbert and Mestre.
 As Mestre pointed out to us on  visit to 
Beijing (summer 2012) our method can be viewed as a special case of results of Ellenberg
\cite{JE}. Actually Ellenberg defines coverings of Riemann surfaces which define algebraic curves
with special endomorphism algebras of their Jacobians, but his method cannot provide explicit equations 
for those curves because of the nonalgorithmic nature of Riemann's Existence Theorem.

We construct explicit 2 - dimensional families of nonhyperelliptic curves, defined over 
the field $\QQ (s, t)$. As an application, we calculate the zeta function of one of these curves. We find that 
the numerators of this zeta factor into 3 quadratic polynomials over the field $\QQ (\zeta _7 + \overline{\zeta}_7)$
(at primes of good reduction),  as is expected from the existence of the extra endomorphisms in the Jacobian.  Conjecturally these quadratic factors define the $L$-function of a modular form.

%\subsection{} \label{SS:intro}

%\subsection{} \label{SS:intro5}

Let $K/\QQ$ be a totally real number field of degree 3. The moduli
space principally polarized abelian varieties $A$ such that
$\mr{End}(A)$ contains a fixed order $R \subset K$ is a Hilbert
modular variety $H_R$ of dimension 3 (for the general theory, see \cite{eG},  \cite{gS}, \cite{gS2}).

As mentioned above, we can
(generically) identify $A = \mr{Jac}(X)$ for a genus 3 curve $X$.
We obtain in this way a morphism $H_R \to  \mfr{M}_3 $. The construction in this 
paper gives explicit families of nonhyperelliptic curves of genus 3 whose Jacobians have 
endomorphisms by $\QQ (\zeta _7 ^+)$. These families have 2 independent moduli and therefore  
do not get the generic member of $H_R$. Nonetheless ours is an explicit algorithm to construct 
families of curves which may be useful in other contexts.  A sequel to this paper will discuss another approach 
which does give families with 3 independent moduli.Ê

The outline of this paper is as follows: In section \ref{S:mconstruction} we review Mestre's method, and give a generalization of it. 
We relate this in section \ref{S:econstruction} to Ellenberg's approach, and in particular show that the Jacobians of these curves 
have a endomorphisms by  $\QQ (\zeta _7 ^+)$. The last section \ref{S:example} looks at an example. We compute the Euler 
factors of the zeta function for some values of the rational prime $p$ of good reduction, and experimentally at least, verify their expected behavior, 
namely the factorization of their numerators into quadratic pieces mentioned above.

Many of the calculations in this paper were carried out with
Mathematica (\cite{Math}), Magma (\cite{Magma}), PARI/GP (\cite{pari}) and Sage (\cite{Sage}).

\noindent {\bf Acknowledgements:}
{We would like to thank Dr. J. F. Mestre for his helpful comments during his visit to Beijing in summer 2012.  This work was
conducted during an invited academic visit to Beijing International Center for Mathematical Research and the Chinese Academy
of Science, and we would like to thank the hosting institutions, in particular Zhibin Liang and Lihong Zhi, for the invitations, and the hospitality during the visit.  The first author is supported in part by NSA grant 115-60-5012 and NSF grant OISE-1318015; the second author is supported by NSFC11001183 and NSFC1171231, and the last author would like to thank the GRFC grant from Southeast Missouri State University.}

\section{Generalization of Mestre's construction}
\label{S:mconstruction} 
We will first briefly review the
construction of Mestre; for details see his papers \cite{jfM} and
\cite{jfM2}.

Let $E$ be an elliptic curve and $G \subset E$ a finite subgroup.
Let $x : E \to \Proj ^1$ be a rational function such that $x(-P) =
x(P)$, for instance the $x$-coordinate in a  Weierstrass model of
$E$.  We also assume that $x(O) = \infty$. If $p: E \to E/G$, then there is a rational function $u : \Proj ^1
\to \Proj ^1$ making the diagram commute:
\[
\begin{CD}
E @>p>> E/G = E'\\
@VxVV @VzVV \\
\Proj ^1 @>u>> \Proj^1
\end{CD}
\] 

These curves have correspondences induced by the elements of $G$. For later purposes, 
we will label the left-hand projective line as $\Proj ^1 _x$, the right-hand projective line 
as $\Proj ^1 _z$ and the projection $z: E' \to \Proj ^1 _z$. 

Any $x\in \Proj^1 _x$ is of the form $x=x(P)$ for a $P\in E$ unique up
to $\pm P$. Because of $z\circ p = u\circ x$ we have $u(x(P+M)) =
u(x(P))$ for any $M\in G$. Therefore the map $X_s \to \mr{Sym}^2
(X_s)$
\[
(x, y)= (x(P), y)\mapsto (x(P+M), y)+(x(P-M), y)
\]
is well-defined, and one can show that it is  a morphism of
varieties.

We define a family of hyperelliptic
curves by
\[
X_s: y^2 = u(x) - s, \, \, \text{where $s$ is a parameter}.
\]
The special case of interest for us is when $G= \langle S \rangle
$ is the subgroup generated by a point of order $7$.  We have the
universal family  $(E, S)$ of elliptic curves with a point of
order $7$ parametrized by the modular curve $X_1(7) \cong \Proj^1 _t$,
given by the equation:
\[
E_t: y^2 + (1 + t - t^2) xy + (t^2 - t^3) y = x^3 + (t^2 - t^3)
x^2, \ \ S = (0, 0).
\]
The subgroup scheme $G = \langle S \rangle$ is defined by the
equation $x(x-t^3+t^2)(x-t^2+t)=0$. The quotient $E/G$ is
\begin{align*}
& y^2 + (-t^2 + t + 1)xy + (-t^3 + t^2)y =
x^3 + (-t^3 + t^2)x^2 + a(t)x + b(t),\\
& a(t) = -5 (-1 + t) t (1 - t + t^2) (1 - 5 t + 2 t^2 + t^3),\\
&b(t) = -(-1 + t) t (1 - 18 t + 76 t^2 - 182 t^3 + 211 t^4 - 132
t^5 +
   70 t^6 - 37 t^7 + 9 t^8 + t^9).
\end{align*}
The map $E\to E/G$ is given by $(x, y) \to (u(x), v(x, y))$ where
\begin{align*}
u(x)= u_t(x) &=  \frac{w(x)}{(-t + t^2 - x)^2 (-t^2 + t^3 - x)^2 x^2} , \\
 w(x)=w_t(x)&=  x^7   -2 (-1 + t) t (1 + t)
   x^6+ (-1 + t) t (1 - 7 t + 5 t^2 - 3 t^3 + 2 t^4 + t^5)x^5 \\
  &  -(-1 + t)^2 t^3 (-1 - 13 t + 12 t^2 - 9 t^3 + 6 t^4) x^4
   +(-1 + t)^3 t^4 (-1 - 7 t - 8 t^2 + 4 t^3 + t^4 + t^5) x^3
  \\
  &  -(-1 + t)^4 t^6 (1 + t) (-3 - 5 t + 3 t^2) x^2
  +(-1 + t)^5 t^8 (-3 - 3 t + t^2) x+(-1 + t)^6 t^{10} .
\end{align*}
Note that the denominator above is
$[(x-x(S))(x-x(2S))(x-x(3S))]^2$. We obtain in this way a 2-parameter family of genus 3 
hyperelliptic curves, which are studied in detail in \cite{HW}.

%\subsection{} \label{SS:gmestre3}

These constructions of Mestre were applied to the hyperelliptic case. In fact, the method is more general.
If $X: f(u, y) = 0$ is any algebraic curve ($f$ is a rational function), then the curve $f(u(x), y)=0$ will
have correspondences induced from elements of the subgroup $G \subset E$. The essential point is the
property $u(x(P+S))= u(x(P))$ for any point $P\in E$ and $S\in G$, which follows immediately from the
definition of $u(x)$.  Hence if $(x(P), y)$ is on $X$ so is $(x(P+S), y)$ . The correspondence is
\[
(x, y) = (x(P), y)\mapsto (x(P+S), y)+(x(P-S), y).
\]
This is well-defined: any $(x, y)$ on $X$ is of the form $(x(P), y)$ for some $P \in E$, unique
up to $\pm P$. Replacing $P$ by $-P$ in the above expression gives the same result. 

If $f(y) = y^2 -t$ for a parameter $t$, Mestre's family is $f(y) - u(x) = 0$. That the genus of a curve 
$X_t$ in this family is $3$
can be seen from the ramification properties of the maps
\[
y\mapsto z = f(y) = y^2-t : \Proj ^1 _y\rightarrow \Proj ^1 _z ,\quad
x\mapsto z = u(x): \Proj ^1 _x \rightarrow \Proj ^1 _z,
\]
see below.

For nonhyperelliptic case, we would like to find an equation $f(y)=u(x)$, where the function $f(y)$ is chosen such
that the curve $f(y)=u(x)$ is of genus 3. The key is in the ramification properties of the map $u$. The morphism 
$u$ has degree 7. Since both source and target have genus 0, the genus 
formula tells us that the total ramification of the map 
$$z = u(x): \Proj ^1 _x \to \Proj ^1 _z$$
is 12. We can identify the ramification points easily.  Let $Q_1 = O, Q_i, i= 2, 3, 4$ be the 
points of order 2 on $E$, and define $q_i  = x(Q_i)$, so $q_1 = \infty$. Since $x(Q+ j S) = x(Q-jS)$
for any multiple of $S$, a generating point of order 7, and any point $Q$ of order 2, we let
$\{ r_i, r_i', r_i '' \}$ be the values $\{x(Q+ j S)    \}$ for $j \neq 0$ mod $7$. Because 
7 is prime to 2, $Q_i ' = f(Q_i)$ are the points of order 2 on $E' = E/\langle P\rangle$. Finally remember that 
the ramification of the maps $x: E \to  \Proj ^1 _x $ and  $z: E' \to  \Proj ^1 _z $ occur only at the points
of order 2, and that the map $f$ is unramified. 

\begin{lemma}
\label{L:ram}
 The map $z = u(x): \Proj ^1 _x \to \Proj ^1 _z$ is ramified of type 2,2,2,1
above each of the four points $q_i, ' = x(Q_i ')$. In fact, $u$ is \'etale at $q_i = x(Q_i)$
and branched of order 2 at each of $r_i, r_i' , r_i''$, as in the picture below.

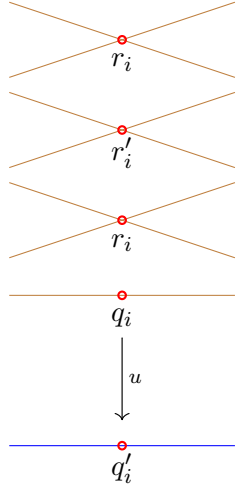
\[
\begin{tikzpicture}

\draw[brown] (0,5.9) --(3,4.9);
\draw[brown] (0,4.9) --(3,5.9);
 \draw [red, thick] (1.5,5.4) circle [radius=0.05];
 \node (P2) at (1.5, 0.1+5) {$r_i $};

\draw[brown] (0,4.7) --(3,3.7);
\draw[brown] (0,3.7) --(3,4.7);
 \draw [red, thick] (1.5,4.2) circle [radius=0.05];
 \node (P2) at (1.5, -0.1+4) {$r_i' $};
 
 \draw[brown] (0,3.5) --(3,2.5);
\draw[brown] (0,2.5) --(3,3.5);
 \draw [red, thick] (1.5,3) circle [radius=0.05];
 \node (P2) at (1.5, -0.3+3) {$r_i $};
 
\draw[brown] (0,2) --(3, 2);
 \draw [red, thick] (1.5,2) circle [radius=0.05];
 \node (P1) at (1.5, -0.3+2) {$q_i $};

\draw[blue] (0,0) --(3, 0);
 \draw [red, thick] (1.5,0) circle [radius=0.05];
 \node (P0) at (1.5, -.3) {$q_i '$};
 \node (R0) at  (1.5, .2) {};
 
 \path[commutative diagrams/.cd, every arrow, every label]
 (P1) edge node {$u$} (R0);
\end{tikzpicture}
\]
\end{lemma}

\begin{proof}
Look at the ramification of the composite map $u\circ x = z\circ f$. Since 
$f$ is everywhere unramified, the branching of this function can only occur at
$f^{-1} (\text{branch points of } z) = f^{-1} (Q_i ') = \{  Q_i + j S) \}$ and these are branch points of order 2.
On the other hand, $Q_i$ is a branch point of order 2 for $x : E \to \Proj ^1 _x$, so there cannot be any further 
branching at this point under $u\circ x$, which means that $u$ must be unramified at $q_i = x(Q_i)$. 
But $x : E \to \Proj ^1 _x$ is unramified at $Q_i + j S$ if $j$ is not 0 mod 7 (these are not points of order 2), 
and since $u\circ x$ does ramify of order 2 at this point and $x$ does not, we must have that $u$ ramifies of order 2 at
each of $r_i, r_i' , r_i''$ as claimed.
\end{proof}

One can understand  Mestre's method as follows: he base-extends a double cover $a  : \Proj  ^1 _y \to \Proj ^1 _z$ 
in the shape $z = f(y) =  y^2 +c$ for an constant $c$, via the map $z = u(x): \Proj ^1 _x \to \Proj ^1 _z$, and normalizes
to obtain a smooth curve $X$. The point $S$ of order 7 on the elliptic curve $E$ gives rise to a correspondence 
which defines an endomorphism of the Jacobian of $X$, as outlined above. The fact that the genus of $X$ is 3
can be immediately seen from the explicit equations, but in fact, it can also be seen from the ramification
properties of the map $g: X \to \Proj ^1 _x $ which is the base-extension of $f  : \Proj  ^1 _y \to \Proj ^1 _z$  via 
$u$. The essential point is that $a$ has ramification index 2 above $q_1' = \infty$ and that $u$ has ramification
type $2,2,2,1$ over that point. Upon base-extension, $g$ ramifies over $q_1, r_1, r_1', r_1''$ respectively of type
$2, 1, 1, 1$.  The point is that, $u$ is unbranched at $q_1$, but branched of order 2 at $ r_1, r_1', r_1''$. Pulling back
along an \'etale map does not change the ramification, but pulling back a double branching along a double branching
``unscrews'' it, as in the picture below. 

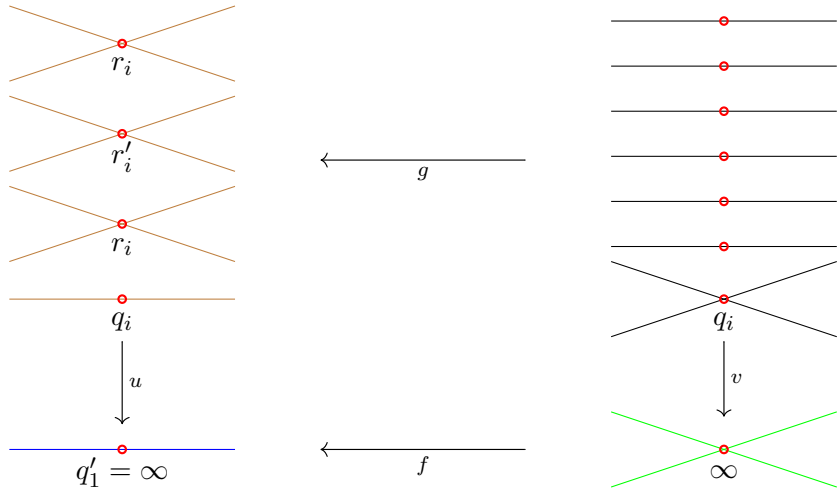
\[
\begin{tikzpicture}

\draw[brown] (0,5.9) --(3,4.9);
\draw[brown] (0,4.9) --(3,5.9);
 \draw [red, thick] (1.5,5.4) circle [radius=0.05];
 \node (P2) at (1.5, 0.1+5) {$r_i $};

\draw[brown] (0,4.7) --(3,3.7);
\draw[brown] (0,3.7) --(3,4.7);
 \draw [red, thick] (1.5,4.2) circle [radius=0.05];
 \node (P2) at (1.5, -0.1+4) {$r_i' $};
 
 \draw[brown] (0,3.5) --(3,2.5);
\draw[brown] (0,2.5) --(3,3.5);
 \draw [red, thick] (1.5,3) circle [radius=0.05];
 \node (P2) at (1.5, -0.3+3) {$r_i $};
 
\draw[brown] (0,2) --(3, 2);
 \draw [red, thick] (1.5,2) circle [radius=0.05];
 \node (P1) at (1.5, -0.3+2) {$q_i $};

\draw[blue] (0,0) --(3, 0);
 \draw [red, thick] (1.5,0) circle [radius=0.05];
 \node (P0) at (1.5, -.3) {$q_1' = \infty$};
 \node (R0) at  (1.5, .2) {};
 
 \path[commutative diagrams/.cd, every arrow, every label]
 (P1) edge node {$u$} (R0);

 \draw[black] (8,5.7) --(11,5.7);
\draw[black] (8,5.1) --(11,5.1);
 \draw [red, thick] (9.5, 5.1) circle [radius=0.05];
  \draw [red, thick] (9.5, 5.7) circle [radius=0.05];

 \draw[black] (8,4.5) --(11,4.5);
\draw[black] (8,3.9) --(11,3.9);
 \draw [red, thick] (9.5, 3.9) circle [radius=0.05];
  \draw [red, thick] (9.5, 4.5) circle [radius=0.05];

\draw[black] (8,2.7) --(11,2.7);
\draw[black] (8,3.3) --(11,3.3);
 \draw [red, thick] (9.5, 2.7) circle [radius=0.05];
  \draw [red, thick] (9.5, 3.3) circle [radius=0.05];

\draw[black] (8,1.5) --(11, 2.5);
\draw[black] (8,2.5) --(11, 1.5);
 \draw [red, thick] (9.5, 2) circle [radius=0.05];

\draw[green] (8,-.5) --(11, .5);
\draw[green] (8,.5) --(11, -.5);
 \draw [red, thick] (9.5,0) circle [radius=0.05];
 \node (Q0) at (9.5, -.3) {$\infty$};
 
  \node (Q1) at (9.5, -0.3+2) {$q_i $};
   \node (Q2) at (9.5, .3) {};
   
 \node (S0) at  (4, 0) {};
  \node (T0) at  (7, 0) {};
  \path[commutative diagrams/.cd, every arrow, every label]
 (T0) edge node {$f$} (S0);

  \node (U0) at  (4, 3.85) {};
  \node (V0) at  (7, 3.85) {};
  \path[commutative diagrams/.cd, every arrow, every label]
 (V0) edge node {$g$} (U0);

  \path[commutative diagrams/.cd, every arrow, every label]
 (Q1) edge node {$v$} (Q2);
 
\end{tikzpicture}
\]

The general case is this:

\begin{lemma}
\label{L:disk}
Let $D_w$ be the open unit disk in the complex $w$-line. Consider the map
$f: D_t \to D_s$ given by $s=f(t) = t^n$ for a positive integer $n$. Let
$g: D_z \to D_s$ be given by $s=g(z) = z^d$ for a positive integer $d$.
Let $e$ be the greatest common divisor of $n$ and $d$. Then the normalization
of the base-change is a disjoint union of $e$ analytic disks $D_w$:
\[
  \widetilde { D_z \times _{D_{s}} D_t}= \coprod _{i = 1} ^{e}D_w.
\]
\end{lemma}
\begin{proof}
The base-change is $\{ (z, t) \mid z^d = t^n \}$. Write $d=e d_0$, $n = e n_0$, with
$d_0$, $n_0$ relatively prime. Choose any primitive $e$th root of unity
$\eta$. This set breaks up into $e$ sets
\[
M_i = \{ (z, t) \mid z^{d_0} = \eta ^i t^{n_0} \}, \quad i = 0, ..., e-1.
\]
These are irreducible and mutually disjoint except for the one common solution $(0, 0)$. The normalization
will be the disjoint union of the normalization of each, and each of these is a disk. A one-to-one parametrization
$D_w\to M_i $ is given by $w \mapsto (\xi w^{n_0}, w^{d_0})$, where $\xi ^{d_0} = \eta ^i$. One checks that it is
a bijection. If $(\xi w^{n_0}, w^{d_0}) = (\xi w_1^{n_0}, w_1 ^{d_0})$, then $(w/w_1)^{d_0} = (w/w_1)^{n_0}$,
which shows that $w/w_1 = 1$ since $d_0$ and $n_0$ are relatively prime, which shows injectivity.
To see surjectivity, for a given $(z, t)$, let $w^{d_0}= t$. Since
\[
(\xi w^{n_0})^{d_0} = \eta ^i t^{n_0} = z^{d_0},
\]
we have $\xi w^{n_0} = \zeta z$, with $\zeta ^{d_0} = 1$. Let $a$ be an integer such that
$a n_0 \equiv -1 $ mod $d_0$, which exists since $d_0$ and $n_0$ are relatively prime.
Let $w_1 = \zeta ^a w$.  Then $w_1 ^{d_0} = t$, and
$z = \zeta ^{-1}\xi w^{n_0}= \zeta ^{-1}\xi \zeta ^{-an_0} w_1^{n_0}= \xi w_1 ^{n_0}$. 
\end{proof}

In summary: the projection   $g: X \to \Proj ^1 _x $ ramifies only once above the points 
$u^{-1} (\infty)$, as in the above picture. The other branch points occur at 
$u^{-1} (\text{the other branch point of } f)$, but $u$ is unbranched over these, so we get
7 branch points of order 2. The total is $e = 1+7=8$, so that the genus of $X$ is indeed 3.

%\subsection{} \label{SS:gmestre5}
We change the notation a bit by setting $q_1 ' = \infty,  q_2 ' = m, q_2' = n, q_3' = p$.
The ramification of the function $u(x)$ is of type $2,2,2,1$ over
the four points $m, n, p, \infty$ which are the images of the
points $x(Q_i)$ for the 4 points $Q_1, Q_2, Q_3, Q_4 \in E$ of
order 2 in the group law of $E$.  

We generalize Mestre's construction by pulling back a triple covering $f  : \Proj  ^1 _y \to \Proj ^1 _z$
along the map $u$. For essentially the same reason as in the hyperelliptic case, we obtain a 
correspondence on this curve $X$. We arrange so that the map $f$ has branching of type $2, 1$
over each of the points $m, n, p$. There is one further branch point of $f$ which is left arbitrary. Because 
of the branching behavior of the map $u$ above the points $m, n, p$ (see the picture in lemma \ref{L:ram})
we get by the reasoning of the previous section that the ramification of the map $g: X \to \Proj ^1 _x $
is of the following type: Above the points $u^{-1}(m),u^{-1}(n),u^{-1}(p) $ only one ramification point in each set. 
Above each of the seven points in $u^{-1} (\text{the other branch point of } f )$, a single double branch point. 
The total ramification is thus $10$ and since $g$ is a triple cover, this gives a genus of 3 for $X$. 
We will show in the next section that the correspondence gives an endomorphism of the Jacobian of $X$
which is $\zeta _7 + \zeta _7 ^{-1}$ by connecting our construction with more general results of Ellenberg. 

%\subsection{} \label{SS:gmestre6}
We obtain a one parameter family $f_s(y)$ of coverings
of degree 3, branched above $m, n, p$. We symmetrize our projection with respect to 
permutations of $m, n, p$ to make our equations rational in the coefficients 
of the polynomial $(x-m)(x-n)(x-p) = x^3+ax^2 +bx +c$, which give the 
$x$-coordinates of the points of order 2 on the elliptic curve $E'$.

{\noindent \bf Procedure to Construct the Curves:}
\begin{itemize}
\item[1.]  Let $T=\dfrac{ay^3+by^2}{y+c}$ with $a \ne 0$.  This
gives a degree 3 covering $\PP^1 \to \PP^1$ with branching above
$T=0, \infty$ with $y= 0,  \infty$ being the branching
points.  Moreover, the ramification of the function is of type $2,
1$ over these points, i.e., $T(0) = T'(0) = 0, T(\infty) = T'(\infty) = \infty$.
 We can choose the constants so that $T(1)=1, T'(1) = 0$. First, we have 
 $c\neq -1$. Then these conditions lead to $a+b=c+1$ and
  $c=\frac{-2a-b}{3a+2b} =a+b-1$. The solution $b=-a$ is rejected because 
  it leads to $c=-1$. The other solution is $b=(1-3a )/2$, and we remove the 
  denominator by replacing $a \to 1+2a$. It leads to the expression 
 
$T=\dfrac{(1+2a)y^3+(-1-3a)y^2}{y-1-a}$ with $a \ne -1/2$. This $T$ carries 
$0, 1, \infty$ respectively to $0, 1, \infty$ with multiplicity 2.

\item[2.] Let $T=\frac{(S-m)(p-n)}{(S-p)(m-n)}$.  This function
carries $T=0, 1, \infty$ to $S=m, n, p$.  Similarly, let
$y=\frac{(z-m)(p-n)}{(z-p)(m-n)}$.  This function carries $y=0, 1,
\infty$ to $z=m, n, p$.  Substitute $T, y$ expression for $T(y)$,  and we obtain
\begin{eqnarray*}
S & = & \frac{\alpha(z)}{\beta(z)}, \, \, \text{ where } \\
\alpha(z) & = & (m^2 n^2 p + 2 a m^2 n^2 p - a m^2 n p^2 - m n^2
p^2 - a m n^2 p^2) \\
& & +( -  3 m^2 n p  - 3 a m^2 n p  - 3 a m n^2 p  + 3 m n p^2  +
 6 a m n p^2 ) z \\
 & &  + ( m^2 n - m n^2  + 2 m^2 p + 3 a m^2 p  +
  n^2 p  + 3 a n^2 p  - 2 m p^2  - 3 a m p^2  -
 n p^2  - 3 a n p^2 ) z^2 \\
 & & +( - m^2 - a m^2  + m n  +
 2 a m n - a n^2  - a m p  - n p  - a n p  +
 p^2  + 2 a p^2 )z^3 \\
 \beta(z) & = & (m^2 n^2 + 2 a m^2 n^2 - m^2 n p - a m^2 n p - a m n^2 p - a m^2 p^2 +
 m n p^2 + 2 a m n p^2 - n^2 p^2 - a n^2 p^2 ) \\
 & & + (- m^2 n  -
 3 a m^2 n  - 2 m n^2  - 3 a m n^2  + m^2 p  + 3 a m^2 p +
 2 n^2 p + 3 a n^2 p  - m p^2 z + n p^2 )z \\
 & &  + (3 m n  +
 6 a m n  - 3 a m p  - 3 n p  - 3 a n p ) z^2 \\
 & & + ( - m  -
 a m  - a n  + p  + 2 a p )z^3.
 \end{eqnarray*}
This $S$ carries 
$m, n, p$ respectively to $m, n, p$ with multiplicity 2.

\item[3.] We need to find the parameter $a$ so that $f$ is
symmetric with respect to $m, n, p$. We do this by making 
$a$ a function of $m, n, p$. To do so, we observe that
\begin{eqnarray*}
f(a, m, n, p, y) &=& f(b, n, m, p,y) \quad  \Rightarrow \quad
b=-1-a, \\
f(a, m, n, p,y) & = & f(b, p, n, m,y) \quad  \Rightarrow \quad
b=\frac{-a}{1+3a}, \\
f(a, m, n, p, y) & = & f(b, m, p, n, y) \quad  \Rightarrow \quad
b=-\frac {1+2a}{2+3a}, \\
f(a, m, n, p,y) & = & f(b, p, m, n,y) \quad  \Rightarrow \quad
b=\frac{-1-a}{2+3a}\\
f(a, m, n, p,y) & = & f(b, n, p, m,y) \quad  \Rightarrow \quad
b=-\frac{1+2a}{1+3a}.
\end{eqnarray*}
In other words, $a$ undergoes a group of linear fractional transformations as above
isomorphic to $S_3$ as we permute $m, n, p$. We can realize this group of linear fractional
transformations by a rational function of degree 1. The reader can check that  
\[
a(m, n, p, s)=\dfrac{(mn-2mp+np)+(m-2n+p)s}{3(m-n)(p-s)}
\]
undergoes the same transformations as above when the $m, n, p$ are permuted, e.g., 
\begin{eqnarray*}
a( n, m, p, s) &=&  -1-a (m, n, p, s) \\
a (n, p, m ,s) & = &
\frac{1+2a(m, n, p, s)}{1+3a(m, n, p, s)}, 
\end{eqnarray*}
and this is sufficient since these permutations generate $S_3$. Define
\[
f_s(y)  = f_s(m, n, p, y) = f(a(m,n,p,s), m,n,p, y).
\]
Then $f_s(m, n, p, y)$ is invariant under all permutations of $m, n, p $, and 
\begin{eqnarray*}
 f_s(y)& =& \dfrac{\alpha_s(y)}{\beta_s(y)}, \, \,   \text{ where }\\
\alpha_s(y) & =& (2 m^3 n^3 p - 2 m^3 n^2 p^2 - 2 m^2 n^3 p^2 + 2
m^3 n p^3 -  2 m^2 n^2 p^3 + 2 m n^3 p^3) \\
& & +( - m^3 n^2 p  - m^2 n^3 p  -  m^3 n p^2  + 6 m^2 n^2 p^2  -
m n^3 p^2 - m^2 n p^3  -  m n^2 p^3) s \\
& & ( - 3 m^3 n^2 p - 3 m^2 n^3 p - 3 m^3 n p^2  +
 18 m^2 n^2 p^2  - 3 m n^3 p^2  - 3 m^2 n p^3  - 3 m n^2 p^3 )y  \\
 & &  + ( 6 m^3 n p - 6 m^2 n^2 p  + 6 m n^3 p - 6 m^2 n p^2  -
 6 m n^2 p^2  + 6 m n p^3 ) s y \\
 & &  + ( 6 m^3 n p  - 6 m^2 n^2 p  +
 6 m n^3 p - 6 m^2 n p^2  - 6 m n^2 p^2  + 6 m n p^3  ) y^2 \\
 & &  + (-
 3 m^3 n + 6 m^2 n^2  - 3 m n^3  - 3 m^3 p -
 3 n^3 p + 6 m^2 p^2  + 6 n^2 p^2  - 3 m p^3  -
 3 n p^3 ) s y^2 \\
 & & +( - m^3 n  + 2 m^2 n^2  - m n^3 - m^3 p  -
 n^3 p  + 2 m^2 p^2  + 2 n^2 p^2  - m p^3  - n p^3 ) y^3 \\
 & &  +( 2 m^3  - 2 m^2 n - 2 m n^2  + 2 n^3 -
 2 m^2 p  + 6 m n p  - 2 n^2 p  - 2 m p^2  -
 2 n p^2  + 2 p^3 ) s y^3 \\
   \beta_s(y) & = & (2 m^3 n^3 - 2 m^3 n^2 p - 2 m^2 n^3 p - 2 m^3 n p^2 + 6 m^2 n^2 p^2 -
 2 m n^3 p^2 + 2 m^3 p^3 - 2 m^2 n p^3 - 2 m n^2 p^3 + 2 n^3 p^3 ) \\
 & &  +( -
 m^3 n^2 - m^2 n^3  + 2 m^3 n p  + 2 m n^3 p  - m^3 p^2  -
 n^3 p^2  - m^2 p^3  + 2 m n p^3  - n^2 p^3 ) s \\
 & & + (- 3 m^3 n^2  -
 3 m^2 n^3  + 6 m^3 n p  + 6 m n^3 p - 3 m^3 p^2 -
 3 n^3 p^2 - 3 m^2 p^3  + 6 m n p^3  - 3 n^2 p^3 ) y \\
 & &  + (
 6 m^2 n^2  - 6 m^2 n p  - 6 m n^2 p  + 6 m^2 p^2  -
 6 m n p^2  + 6 n^2 p^2) s y  \\
 & & + ( 6 m^2 n^2  - 6 m^2 n p  -
 6 m n^2 p  + 6 m^2 p^2  - 6 m n p^2  + 6 n^2 p^2 )y^2 \\
 & & +( -
 3 m^2 n  - 3 m n^2  - 3 m^2 p  + 18 m n p  -
 3 n^2 p  - 3 m p^2  - 3 n p^2 ) s y^2 \\
 & & +( - m^2 n -
 m n^2  - m^2 p  + 6 m n p  - n^2 p - m p^2  -
 n p^2 ) y^3 \\
 & &  + (2 m^2 - 2 m n  + 2 n^2  - 2 m p  -
 2 n p  + 2 p^2 )s y^3.
\end{eqnarray*}

\item[4.] In terms of the elementary symmetric functions of $m,n, p$ this is:
\begin{eqnarray*}
 f_s(y)& =& \dfrac{\alpha_s(y)}{\beta_s(y)}, \, \,   \text{ where }\\
\alpha_s(y) & =& y^3 \left(s \left(2 \sigma _1^3-8
   \sigma _2 \sigma _1+18 \sigma _3\right)-\sigma
   _2 \sigma _1^2-3 \sigma _3 \sigma _1+4 \sigma
   _2^2\right)   \\
& &+y^2 \left(s \left(-3 \sigma _2
   \sigma _1^2-9 \sigma _3 \sigma _1+12 \sigma
   _2^2\right)+6 \sigma _3 \sigma _1^2-18 \sigma
   _2 \sigma _3\right) \\  
 & &+y \left(s \left(6 \sigma
   _1^2 \sigma _3-18 \sigma _2 \sigma
   _3\right)+27 \sigma _3^2-3 \sigma _1 \sigma _2
   \sigma _3\right) \\
& &+s \left(9 \sigma _3^2-\sigma _1 \sigma _2 \sigma
   _3\right) -6 \sigma _1 \sigma _3^2+2
   \sigma _2^2 \sigma _3\\
 \\
   \beta_s(y) & = &y^3
   \left(s \left(2 \sigma _1^2-6 \sigma
   _2\right)-\sigma _1 \sigma _2+9 \sigma
   _3\right)\\
   & &+y^2 \left(s \left(27 \sigma _3-3
   \sigma _1 \sigma _2\right)+6 \sigma _2^2-18
   \sigma _1 \sigma _3\right)\\
   & & +y \left(s \left(6
   \sigma _2^2-18 \sigma _1 \sigma _3\right)+12
   \sigma _3 \sigma _1^2-3 \sigma _2^2 \sigma
   _1-9 \sigma _2 \sigma _3\right)\\
 & & +s \left(4 \sigma _3 \sigma _1^2-\sigma _2^2
   \sigma _1-3 \sigma _2 \sigma _3\right)+2 \sigma
   _2^3+18 \sigma _3^2-8 \sigma _1 \sigma _2
   \sigma _3
\end{eqnarray*}

\item[5.] We apply this to the universal family of elliptic curves with a point of order 7
considered in section 2. The $m, n, p$ are the $x$-coordinates of the points of order 2 other than $O$ on the 
quotient curve $E'$. In other words, they are the roots of the 
discriminant with respect to $y$ of the Weierstrass equation for $E'$ given  there. The symmetric 
functions $\sigma _1, \sigma_2, \sigma _3$ are up to signs the coefficients of this discriminant. We get
\begin{eqnarray*}
\sigma _1 &=&\frac{1}{4} (-1 - 2 t - 3 t^2 + 6 t^3 - t^4)\\
\sigma_2 & = &\frac{1}{4} \left(-20 t^7+142 t^5-284 t^4+280
   t^3-138 t^2+20 t\right)\\
\sigma_3 &=&   \frac{1}{4} \left(4 t^{11}+32 t^{10}-184 t^9+428
   t^8-808 t^7+1371 t^6-1570 t^5+1031 t^4-376
   t^3+76 t^2-4 t\right)
\end{eqnarray*}

\end{itemize}

Hence, we have a 2 dimensional family of curves $X_{s,t}$ defined
by $f_s(y)=u_t(x)$, that is $$\alpha_s(y) (-t + t^2 - x)^2 (-t^2 +
t^3 - x)^2 x^2 = w_t(x) \beta_s(y).$$
We will see in the next section that their Jacobians  have endomorphisms by 
the cubic field $\QQ (\zeta _7 + \overline{\zeta}_7)$.

\section{Relation to Ellenberg's construction}
\label{S:econstruction}
%\subsection{}\label{SS:ellen1}
Let $D_7 = \langle \sigma, \tau \mid \sigma ^7 = \tau ^2 = e, \tau \sigma \tau ^{-1} = \sigma ^{-1} \rangle $ be the dihedral group
of order 14, symmetries of a regular $7$-gon. It has $7$ involutions $\tau _i = \sigma ^i \tau \sigma ^{-i}, i =0, ..., 6$.  Corresponding to 
the lattice of subgroups of $D_7$ we have the tower of curves and morphisms as in the following diagram.
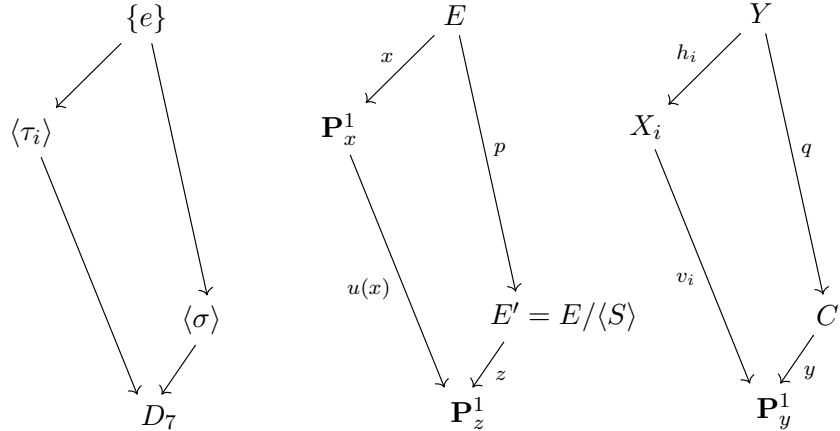
\[
\begin{tikzpicture}
  \node (P0) at (90:2.8cm) {$\{ e\}$};
  \node (P1) at (90+50: 2cm) {$\langle \tau _i  \rangle$} ;
  \node (P2) at (90+ 180: 2.5cm) {\makebox[5ex][r]{$D_7 $}};
  \node (P3) at (90+3*72:1.5cm) {\makebox[5ex][l]{$ \langle  \sigma \rangle$}};
  
  \path[commutative diagrams/.cd, every arrow, every label]
    (P0) edge node[swap] {} (P1)
    (P1) edge node[swap] {} (P2)
    (P3) edge node{} (P2);
    \path[commutative diagrams/.cd, every arrow, every label]
     (P0) edge node {} (P3);
   \end{tikzpicture}\quad\quad
\begin{tikzpicture}
  \node (P0) at (90:2.8cm) {$E$};
  \node (P1) at (90+50: 2cm) {$\Proj  _{x}^{1}$} ;
  \node (P2) at (90+ 180: 2.5cm) {\makebox[5ex][r]{$\Proj  _{z}^{1} $}};
  \node (P3) at (90+3*72:1.5cm) {\makebox[5ex][l]{$E' = E/\langle S \rangle$}};
  
  \path[commutative diagrams/.cd, every arrow, every label]
    (P0) edge node[swap] {$x$} (P1)
    (P1) edge node[swap] {$u(x)$} (P2)
    (P3) edge node{$z$} (P2);
    \path[commutative diagrams/.cd, every arrow, every label]
     (P0) edge node {$p$} (P3);
   \end{tikzpicture}   \quad\quad
  \begin{tikzpicture}
  \node (P0) at (90:2.8cm) {$Y$};
  \node (P1) at (90+50: 2cm) {$X_i$} ;
  \node (P2) at (90+ 180: 2.5cm) {\makebox[5ex][r]{$\Proj  _{y}^{1} $}};
  \node (P3) at (90+3*72:1.5cm) {$C$};
  
  \path[commutative diagrams/.cd, every arrow, every label]
    (P0) edge node[swap] {$h_i$} (P1)
    (P1) edge node[swap] {$v_i$} (P2)
    (P3) edge node{$y$} (P2);
    \path[commutative diagrams/.cd, every arrow, every label]
     (P0) edge node {$q$} (P3);
   \end{tikzpicture} 
   \]
   Here $E$ is an elliptic curve and $S\in E$ is a point of order $7$ as in Mestre's constrtuction. $f$ is the 
   canonical projection, which is an \'etale cyclic covering of degree $7$.  The curves $E/\langle  \tau _i \rangle $
   are all abstractly $\Proj ^1$ and we choose one coordinate $x$ for all of them. We do this as follows: 
   We let $x=x_1 : E \to \Proj ^1$ be projection modulo the involution $\tau _1 = \tau:P \mapsto -P$, then $x_i$ is the projection
   modulo the involution $\tau _i : P \mapsto -P + i S$. The diagram on the right results by base-change along the map
   $z \mapsto  u = z^2 +c: \Proj _z ^1  \to \Proj _u ^1$ for a constant $c$. Strictly speaking this is base-change followed by  
normalization. However, the map $q$ really is the base-change of $p$ via the 
map $C \to E'$, since $p$ is \'etale, and no normalization is  required. It follows that $q$ is an \'etale cyclic covering of degree 
$7$. We can see that the curves $X_i = Y/\langle \tau _i  \rangle$, all isomorphic, are of genus 3, the curve $C$ is of genus 2 and the curve $Y$ is of genus 8.   

Here is a diagram:

\[   
  \begin{tikzcd}  %%[row sep=scriptsize, column sep=scriptsize]
& Y \arrow[dotted]{dl}{h}  \arrow{rr}{l}\arrow[near end]{dd}{q} & & E \arrow{dl}{x}\arrow{dd}{p} \\
X \arrow[crossing over]{rr}{\phantom{hhhh}g}\arrow{dd}{v} & & \Proj ^1 _x \\
& C \arrow{dl}{y}\arrow{rr}{\!\!\!\!\!\!\!\!\!\!\!\!\!  \!\!\!\!\!\!\!\!\!\!\!\!\!   k} & & E' \arrow{dl}{z} \\
\Proj ^1 _{y}\arrow{rr}{f} & & \Proj ^1 _{z}\arrow[crossing over, leftarrow, near end, swap] {uu}{ u } \\
\end{tikzcd}
 \]  
The front, bottom and back faces are defined as pull-hack squares. More precisely, they are pull-backs followed by normalization. 
Actually the back face is just a pull-back as mentioned, because the map $p$ is an \'etale map of smooth curves, so its base-change
gives an \'etale map of smooth curves $q$.  The dotted arrow $h$ exists commuting the diagram. In view of universal properties 
of normalization and fiber-products, it suffices to verify that $f\circ (y\circ q)= u\circ (x\circ l)$, which is immediate. Note also that 
this diagram clearly shows an action of $D_7$ on the curve $Y$. The maps $P \mapsto P+S$ on $E$ and the identity on $C$ lift 
to give an automorphism $\sigma$ of order 7  on $Y$ by base-change. The map $P \mapsto -P$ on $E'$ and the identity on 
$\Proj ^1 _y$ lifts to the hyperelliptic involution on $C$. In turn, the hyperelliptic involution on $C$ and the identity map on $E$
lifts to an automorphism $\tau$ of order 2. These generate a dihedral group of symmetries. The ramifications in this picture conform 
to Ellenberg's branching for this example: The overall projection $Y \to \Proj ^1_y$ is a Galois $D_7$ covering branched above 
6 points of branching type $2,2,2,2,2,2,2$. The six branch points $\Sigma \subset \Proj ^1 _y$ are the branch points of the 
projection of the genus 2 curve $C \to \Proj ^1 _y$. Thus the branched covering 
$a: Y(\C) \to \Proj ^1_y (\C)$
is topologically described by the monodromy homomorphism
\[
\rho : \pi _1 (\Proj ^1 _y (\C)\setminus \Sigma, *) \to \mr{Aut}(a^{-1} (*))\sim S_7
\] 
acting transitively. The image of $\rho$ is isomorphic to $D_7$. That the action is a transitive
Galois action  means that, as 
$D_7$-set, $a^{-1} (*)$ is isomorphic to $D_7$ with its natural $D_7$-action by (right) translation.
 Finally, this image is generated 
by the 6 involutions $\tau _i \in D_7$ gotten by walking from $*$ around each of the 6 points in $\Sigma$. These 
satisfy $\tau _1 \tau _2\tau _3\tau _4\tau _5\tau _6 = 1$. In terms of the generators $\sigma, \tau$ of $D_7$, 
this means that $\tau_i = \sigma ^{b_i} \tau$, $i = 1, ..., 6$.  It is not difficult to see that the action is transitive if and only if
at least two of the $b_i$ are distinct modulo 7. 

Recall the general set-up in Ellenberg's paper. We let $Y$ be a smooth projective curve 
an $G$ be a finite group acting on $Y$. There is a homomorphism of algebras 
\[
\ZZ [G]\rightarrow \text{End}(\text{Jac}(Y))
\]
and similar maps to endomorphism algebras of $H^i(Y)$ for various cohomology theories 
$H^i$. After $\otimes \QQ$, both sides above are semisimple algebras of finite dimension. 
If $H \subset G$ is a subgroup, we have the idempotent
\[
\frac{1}{\# H} \sum _{h \in H} h   :=\pi _H\in \QQ [G].
\]
One defines a ``Hecke algebra'' $\QQ [H \backslash G /H]$ as the subalgebra of
$\QQ [G]$ generated by $\pi _H g \pi _H$ for all $g \in G$. If $X = Y/H$, then we have a homomorphism 
\[
\QQ [H \backslash G /H] \rightarrow   \text{End}(\text{Jac}(X))
\]
since $\text{Jac}(X) = \pi _H \text{Jac}(Y)$. Similar results apply to endomorphisms of
$H^i(X)$.

Our case is the genus 8 curve $Y$ with the action of $G = D_7$ described above. The 
curve $X = Y/H$ has genus 3, where $H = \langle \tau \rangle$ is generated by any involution. One can show 
directly that $\QQ [H \backslash G /H] $ is isomorphic to $\QQ ([\zeta _7 + \overline{\zeta }_7]$, but 
in any case, this is one of a family of examples studied in Ellenberg's paper.  

In view of the classical isomorphism $\text{End}(\text{Jac}(X)) = \text{Corr}(X)$ of the 
ring of endomorphisms of the Jacobian with the ring of correspondence classes of $X$, for any curve 
$X$, we can describe geometrically the action of the generator $\pi _H \sigma \pi _H$ as follows:
if $h: Y\to X = Y/H$ is the canonical projection, we get a second projection 
$h \circ \sigma : Y \to X$. These two projections define the correspondence, which therefore
has degrees $(2,2)$. 

Thus we have shown our main result:

\begin{theorem}
\label{T:main} Let $X_{s, t}$ be the curve defined by $f_s(y) =
u_t(x) $. For general values of $(s, t)$ this is a genus 3 nonhyperelliptic curve
and there is an endomorphism $\phi$ of the Jacobian of $X_{s, t}$
which satisfies the equation $\phi^3+\phi^2-2\phi-1=0$. The
endomorphism is defined over the field $\QQ (s, t)$.
\end{theorem}

\section{An example.}
\label{S:example}
Let $X$ be the smooth projective model of the curve $X_{s, t}$ with $s=0, t=-2$. The calculations that follow were done in Magma
and PARI/GP in a
Sage worksheet. This curve is irreducible and nonhyperelliptic of genus 3. The canonical model of this curve is the projective plane 
quartic:
\begin{align*}
&x^4+\frac{345 x^3 y}{4}-\frac{16038 x^3 z}{7}+\frac{14499 x^2 y^2}{14}-\frac{553623}{4} x^2 y
   z+\frac{4273137 x^2 z^2}{2}+\frac{2153679 x y^3}{28}+\frac{28315359}{7} x y^2 z\\
 & +\frac{659015811}{7} x y z^2-\frac{6866481456 x z^3}{7}-\frac{28405935 y^4}{7}-20973087 y^3
   z-\frac{10692058320 y^2 z^2}{7}\\
   &-\frac{205496736912 y z^3}{7}
   +\frac{1321162646760 z^4}{7}=0.
\end{align*}

We compute the zeta function of the scheme $X/\ZZ$. This means that we compute
\begin{align*}
Z(X/\FF_p, x)&=\exp \left (\sum _{\nu \ge 1} N_{\nu}x^{\nu}/\nu
\right )= \frac{1+a_px + b_p
x^2+c_px^3+pb_px^4+p^2a_px^5+p^3x^6}{(1-x)(1-px)}
\end{align*}
for the primes $p\not= 2,3,7, 73, 109, 829, 967$  (divisors of the discriminant of the ternary quartic above), where $N_{\nu} = \# X(\FF_{p^{\nu}})$. Since the curve has genus 3 we need only compute
$N_{\nu}$ for $\nu = 1,2,3$. The numerator in the above expression equals
\[
h_p (x) :=
\det\left (
1 - x \rho (\mr{Frob}_p) \mid H^1 _{\mr{et}} (X\otimes \overline{\QQ}, \QQ _l)
      \right ), \ \ l \not= p,
\]
where $\rho : \mr{Gal} (\overline{\QQ}/\QQ)\to \mr{GSp}(H^1 _{\mr{et}} (X\otimes \overline{\QQ}, \QQ _l))$ is the canonical
Galois representation in \'etale cohomology, and $\mr{Frob}_p$ is a Frobenius element at $p$. Since the Jacobian of $X$
has endomorphisms in the field $K = \QQ (\zeta _7 + \overline{\zeta }_7)$, and these endomorphisms are defined over $\QQ$,
$ V_l = H^1 _{\mr{et}} (X\otimes \overline{\QQ}, \QQ _l)$ becomes a free $K_{l} = K\otimes \QQ _l$-module of rank 2, and the operators
in $K$ commute with those in $\mr{Gal} (\overline{\QQ}/\QQ)$, so that the representation $\rho$ factors as
\[
\rho :  \mr{Gal} (\overline{\QQ}/\QQ)\to \mr{GL}_2 (V_l) \subset \mr{GSp}_6 (\QQ _l).
\]
This implies that the characteristic polynomials $h_p(x)$ factor
as $g_p(x) g_p^{\sigma}(x) g_p^{\sigma ^2}(x)$ for a quadratic
polynomial $g_p (x) \in O_K[x]$, where
$O_K = \ZZ [t]/(t^3+t^2-2t-1)$ is the ring of integers of $K$ and $\sigma$ generates the
Galois group of $K$ over $\QQ$. These polynomials, as well as the
trace of $\mr{Frob}_p$ are displayed in Table \ref{T:charpolys}. 
\begin{table}[h]
\begin{tabular}{|c|c|c|}
\hline
$p$ &  $g_p (x)$ &$ \mr{Trace}$ \\
\hline
5 & $1 - t x + 5 x^2$ & $-1$ \\
%\hline 7 & $X^2 + \sqrt{-3}\left (\tfrac{-1-\sqrt{-3}}{2}\right ) X- \left (\tfrac{-1+\sqrt{-3}}{2}\right )7^2$& $-\sqrt{-3}$ & $\omega ^4$ \\
\hline
11 & $1 - t x + 11 x^2$ & $-1$ \\
\hline 13 & $1 + (3 - t) x + 13 x^2$ & $-10$  \\
\hline
17 & $1 + (-1 - 4 t) x + 17 x^2$ & $-1$  \\
\hline 19 & $1 + (6 - 3 t - 2 t^2) x + 19 x^2$ & $-11$ \\
 \hline 23 & $1 + (8 - t - 3 t^2) x + 23 x^2$ & $-10$  \\
 \hline 29 & $1 + (8 - 5 t - 6 t^2) x + 29 x^2 $& $1$ \\
\hline 31 & $1 + (7 - t - 2 t^2) x + 31 x^2 $ & $-12 $ \\
\hline 37 & $1 + (6 - 4 t - 5 t^2) x + 37 x^2$& $3$ \\
\hline41 & $1 + 8 x + 41 x^2 $ & $-24$\\
\hline 43 & $1 + (4 - t - 2 t^2) x + 43 x^2$& $ -3$ \\
\hline47 & $1 + (10 - t - 4 t^2) x + 47 x^2 $ & $ -11$ \\
\hline 53 & $ 1 + (6 + 2 t - 5 t^2) x + 53 x^2$ & $9 $ \\
\hline 59 & $1 + (10 - 6 t - 9 t^2) x + 59 x^2 $ & $9 $ \\
\hline 61 & $ 1 + (-2 + 3 t) x + 61 x^2  $ & $9$  \\
\hline 67 & $1 + (4 - t - 2 t^2) x + 67 x^2  $& $ -3$ \\
\hline 71 & $ 1 + (10 - 4 t - 5 t^2) x + 71 x^2$ & $-9 $  \\
%\hline 73 & $1 - 3 x + 3 x^2 - x^3 + 219 x^4 - 15987 x^5 + 389017 x^6$& $3$ \\
\hline 79 & $1 + (7 - 8 t - 9 t^2) x + 79 x^2 $ & $16 $ \\
\hline 83 & $ 1 + (1 - 3 t - 6 t^2) x + 83 x^2 $ & $ 24$ \\
\hline 89 & $1 + (19 - t - 11 t^2) x + 89 x^2$ & $-3$  \\
%\hline 97 & $x^2 -25 \left (\tfrac{1-\sqrt{-3}}{2}\right ) x
%         + \left (\tfrac{-1-\sqrt{-3}}{2}\right )97^2$
%           & $25$ & $\omega  $ \\
%\hline 101 & $x^2 +72\sqrt{-2} x -101^2$
%           & $-72\sqrt{2}$ & $i$ \\
%\hline 103 & $x^2 - 49\sqrt{-3} \left (\tfrac{-1+\sqrt{-3}}{2}\right
%)x
%            - \left (\tfrac{-1-\sqrt{-3}}{2}\right ) 103^2$
%            & $-49\sqrt{-3}$ & $\omega ^2$ \\
%\hline
%107 & $x^2 +30\sqrt{-6} x - 107^2$ & $-30\sqrt{-6}$ & $1$ \\
%\hline 109 & $x^2 -146  x + 109^2$
%           & $-146$ & $\omega ^3 $ \\
%\hline 157 & $x^2 +190  x + 157^2$
%           & $-190$ & $\omega ^6 $ \\
%\hline 229 & $x^2 -170  x + 229^2$
%           & $-170$ & $\omega ^3 $ \\
%\hline 277 & $x^2 +94  x + 277^2$
%           & $94$ & $\omega ^3 $ \\
\hline
\end{tabular}
\caption{ Factorization of $ h_p(x)= g_p(x) g_p^{\sigma}(x) g_p^{\sigma ^2}(x)$, trace of $\mr{Frob}_p $ at good primes.} \label{T:charpolys}
\end{table}
%%%%%%%%%%%%%%%%%%%%%%%%%%%%%%%%%%%%%%%%%%%%%%%%%%%%%%%%%%%%%%
%%%%%%%%%%%%%%%%%%%%%%%%%%%%%%%%%%%%%%%%%%%%%%%%%%%%%%%%%%%%%%
%%%%%%%%%%%%%%%%%%%%%%%%%%%%%%%%%%%%%%%%%%%%%%%%%%%%%%%%%%%%%%
%%%%%%%%%%%%%%%%%%%%%%%%%%%%%%%%%%%%%%%%%%%%%%%%%%%%%%%%%%%%%%

Using the solution to Serre's conjecture, \cite{KW} it follows that these quadratic factors 
coincide with the Hecke polynomials of modular form, modulo $l$. In fact, one can show that these quadratic factors coincide with the Hecke polynomials of a weight 2 newform.


\newpage

\begin{thebibliography}{10}
\bibitem{Magma}
W. Bosma, J. Cannon, and C. Playoust, {\em The Magma algebra
system. I. The user language}, J. Symbolic Comput., 24 (1997),
235--265.


\bibitem{EK}
Noam Elkies and Abhinav Kumar, {\em K3 surfaces and equations for Hilbert modular surfaces}, arXiv:1209.3527v1, (2012).

\bibitem{JE}
J. Ellenberg, {\em Endomorphism algebras of Jacobians}, Advances
in Mathematics, Vol. 162, No. 2 (2001), 243--271.



\bibitem{eG}
Eyal Z. Goren, Lectures on Hilbert modular varieties and modular forms, With the assistance of Marc-Hubert Nicole, 
CIRM Monograph Series, 14, Amer. Math. Soc. 2002.

\bibitem{dG}
David Gruenewald, {\em Explicit algorithms for Humbert surfaces.} thesis, U. Sydney (2009),
http://echidna.maths.usyd.edu.au/$\sim$davidg/ .


\bibitem{HW}
J. W. Hoffman and H. Wang, {\em 7-gons and genus 3 hyperelliptic
curves},  Revista de la Real Academia de Ciencias Exactas, F'sicas
y Naturales. Serie A. Matem\`aticas, 107 (2013), 35-52.


\bibitem{KW}
 Chandrashekhar Khare and  Jean-Pierre Wintenberger, {\em On Serre's conjecture for 2-dimensional mod p representations of
 $Gal(\overline{Q}/Q)$}. Ann. of Math. (2) 169 (2009), no. 1, 229--253.








\bibitem{jfM}
J. F. Mestre, {\em Courbes hyperelliptiques {\`a} multiplications
r{\'e}elles}, C. R. Acad. Sci. Paris, S«er. I Math., 307 (1988),
721Ð-724.

\bibitem{jfM2}
J. F. Mestre, {\em Courbes hyperelliptiques {\`a} multiplications
r{\'e}elles}, Progr. Math., 89 (1991), 193-Ð208.

\bibitem{dM}
D. Mumford, {\em Abelian Varieties}, Tata Institute of Fundamental
Research Studies in Mathematics, No. 5, Published for the Tata
Institute of Fundamental Research, Bombay; Oxford University
Press, London 1970.


\bibitem{pari}
The PARI~Group, {\em PARI/GP, version {\tt 2.5.0}}, 2011,
Bordeaux, available from http://pari.math.u-bordeaux.fr/

\bibitem{bR}
Bernhard Runge, {\em  Endomorphism rings of abelian surfaces and projective models of their moduli spaces.} Tohoku Math. J. (2) 51 (1999), no. 3, 283--303.


\bibitem{Sage}
SAGE Mathematics Software, Version 4.6, http://www.sagemath.org/

\bibitem{yS}
Y. Sakai, {\em Construction of genus two curves with real
multiplication by Poncelet's theorem.} (2010) dissertation, Waseda
University.



\bibitem{gS}
G. Shimura, {\em On analytic families of polarized abelian
varieties and automorphic functions.} Ann. of Math. (2) 78 1963
149--192.

\bibitem{gS2}
\bysame, Abelian varieties with complex multiplication and modular
functions. Princeton Mathematical Series, 46. Princeton University
Press, Princeton, NJ, 1998.


 \bibitem{aW}
 A. Weil, {\em  Vari{\'e}t{\'e}s ab{\'e}liennes et courbes alg{\'e}briques},
 Hermann \& Cie., Paris, 1948.

\bibitem{Math}
Wolfram Research, Inc., {\sc Mathematica}, Version 7.0, Champaign,
IL (2008).

\end{thebibliography}
\end{document}